\documentclass{birkjour}

\RequirePackage[strict]{csquotes}
\RequirePackage[initials,non-compressed-cites]{amsrefs}

\usepackage{upref}

\usepackage{amsmath,amsthm,amscd,amssymb,siunitx}
\usepackage{braket,mleftright}
\DeclareMathAlphabet\mathrsfso{U}{rsfso}{m}{n}

\usepackage{ragged2e}
\justifying
\fussy

\newcounter{smalllist}
\newenvironment{SL}{\begin{list}{{\rm\roman{smalllist})}}{%
\setlength{\topsep}{0mm}\setlength{\parsep}{0mm}\setlength{\itemsep}{0mm}%
\setlength{\labelwidth}{2em}\setlength{\leftmargin}{2em}\usecounter{smalllist}%
}}{\end{list}}

\newcommand{\ie}{\textit{i.e.}\;}
\newcommand{\eg}{\textit{e.g.}\;}
\newcommand{\cf}{\textit{cf.}\;}

\newcommand{\bbC}{\mathbb{C}}

\newcommand{\what}[1]{\widehat{#1}}		%wide hat
\newcommand{\mrm}[1]{\mathrm{#1}}		%roman math
\newcommand{\co}{\colon}				%colon
\renewcommand{\vrt}{\,\vert\,}			%vertical line
\newcommand{\lto}{\rightarrow}			%space map
\newcommand{\abs}[1]{\lvert#1\rvert}	%absolute value
\newcommand{\norm}[1]{\lVert#1\rVert}	%norm
\newcommand{\img}{\mrm{i}}				%imaginary unit
\newcommand{\op}{\oplus}				%orthogonal sum
\newcommand{\om}{\ominus}				%orthogonal subtraction
\newcommand{\hsum}{\,\what{+}\,}		%componentwise sum
\newcommand{\hop}{\,\what{\op}\,}		%orthogonal componentwise sum
\newcommand{\dsum}{\dotplus}			%direct sum

\DeclareMathOperator{\im}{Im}   %imaginary part
\DeclareMathOperator{\Ind}{Ind}   %multivalued part
\DeclareMathOperator{\Ker}{Ker}   %kernel

\newcommand{\msH}{\mathrsfso{H}}
\newcommand{\msD}{\mathrsfso{D}}
\newcommand{\msR}{\mathrsfso{R}}
\newcommand{\msN}{\mathrsfso{N}}
\newcommand{\msL}{\mathrsfso{L}}
\newcommand{\msB}{\mathrsfso{B}}
\newcommand{\msE}{\mathrsfso{E}}

\newcommand{\vp}{\varphi}

\newcommand{\hx}{\hat{x}}
\newcommand{\hy}{\hat{y}}
\newcommand{\hu}{\hat{u}}
\newcommand{\hv}{\hat{v}}

\newcommand{\ty}{\tilde{y}}
\newcommand{\tV}{\tilde{V}}

\newcommand{\setm}{\smallsetminus}		%%set minus

\newtheorem{thm}{Theorem}%[section]
\newtheorem{lem}[thm]{Lemma}
\newtheorem{prop}[thm]{Proposition}
\theoremstyle{remark}
\newtheorem{rem}[thm]{Remark}
\newtheorem{rems}[thm]{Remarks}

\begin{document}

\title{On a boundary pair
of a dissipative operator}

\author{Rytis Jur\v{s}\.{e}nas}

\address{Vilnius University,
Institute of Theoretical Physics and Astronomy,
Saul\.{e}tekio ave.~3, 10257 Vilnius, Lithuania}

\email{rytis.jursenas@tfai.vu.lt}



\date{\today}

\begin{abstract}
The aim of this brief note is to demonstrate that
the boundary pair of a dissipative
operator is determined by the unitary boundary
pair of its symmetric part.
\end{abstract}

\maketitle
% %%%%%%%%%%%%%%%%%%%%%%%%%%%%%%%%%%%%%%%%%%%%%%%%%%%%%%%%%%%%%%
% %%%%%%%%%%%%%%%%%%%%%%%%%%%%%%%%%%%%%%%%%%%%%%%%%%%%%%%%%%%%%%
% %%%%%%%%%%%%%%%%%%%%%%%%%%%%%%%%%%%%%%%%%%%%%%%%%%%%%%%%%%%%%%
% %%%%%%%%%%%%%%%%%%%%%%%%%%%%%%%%%%%%%%%%%%%%%%%%%%%%%%%%%%%%%%
\section{Introduction}
% %%%%%%%%%%%%%%%%%%%%%%%%%%%%%%%%%%%%%%%%%%%%%%%%%%%%%%%%%%%%%%
% %%%%%%%%%%%%%%%%%%%%%%%%%%%%%%%%%%%%%%%%%%%%%%%%%%%%%%%%%%%%%%
% %%%%%%%%%%%%%%%%%%%%%%%%%%%%%%%%%%%%%%%%%%%%%%%%%%%%%%%%%%%%%%
% %%%%%%%%%%%%%%%%%%%%%%%%%%%%%%%%%%%%%%%%%%%%%%%%%%%%%%%%%%%%%%
Let $T$ be a maximal dissipative operator
in a Hilbert space
$(\msH,\braket{\,\,\cdot\,\,,\,\,\cdot\,\,})$.
According to \cite[Lemma~3.1]{Brown20} there is a Hilbert space
$(\msE,\braket{\,\,\cdot\,\,,\,\,\cdot\,\,}_\msE)$
and
an operator $\Gamma_{01}\co\msH\lto \msE$,
$\msD_{\Gamma_{01}}=\msD_T$, which is bounded in graph norm
of $T$, has dense range in $\msE$,
and such that
\begin{equation}
\braket{x\,,\,Ty}-\braket{Tx\,,\,y}=\img
\braket{\Gamma_{01} x\,,\,\Gamma_{01} y}_\msE
\quad(x,y\in\msD_T)\,.
\label{eq:bvs}
\end{equation}
And similarly for the maximal dissipative
operator $(-T^*)$.

Following \cite{Arlinskii24,Arlinskii23}
the pair $(\msE,\Gamma_{01})$ is a boundary pair of $T$.

Here we show that:
\begin{SL}
\item[a)]
The maximality condition and the density of
the domain of $T$ are not needed for
\eqref{eq:bvs} to hold.
\item[b)]
$(\msH,\braket{\,\,\cdot\,\,,\,\,\cdot\,\,})$
can be replaced by the Krein space
$(\msH,[\,\,\cdot\,\,,\,\,\cdot\,\,])$.
\item[c)]
Up to an isometric isomorphism $\msE$
is determined by the closure of the sesquilinear
form $\gamma_N$.
\item[d)]
Given $\msE$,
the operator $\Gamma_{01}$ is defined
uniquely by the
boundary value space of the symmetric part of $T$.
\end{SL}

With $T$ in \eqref{eq:bvs} one associates the
symmetric positive and closable form
\[
\gamma_T[x\,,\,y]=-\img
( \braket{x\,,\,Ty}-\braket{Tx\,,\,y} )
\quad\text{on}\quad\msD_T\,.
\]
The quadratic form
$\gamma_T[x]=\gamma_T[x\,,\,x]$. Its kernel,
$\{x\in\msD_T\vrt\im\braket{x\,,\,Tx}=0\}$,
defines the symmetric part $S$ of $T$,
whose orthogonal complement in $T$ is the dissipative
restriction $N=T\vrt(T+\img I)^{-1}(\msN_\img)$, where
$\msN_\img=\Ker_\img S^*$ is the deficiency subspace.
Since the positive closable form
$\gamma_N=\gamma_T$ on $\msD_N$ has trivial kernel,
$(\msD_N,(\,\,\cdot\,\,,\,\,\cdot\,\,)_N)$
with scalar product
$(\,\,\cdot\,\,,\,\,\cdot\,\,)_N=
\gamma_N[\,\,\cdot\,\,,\,\,\cdot\,\,]$
becomes an inner product space; its completion
$(\msD_{\bar{\gamma}_N},(\,\,\cdot\,\,,\,\,\cdot\,\,)^\sim_N)$,
where $\bar{\gamma}_N$ is the closure of $\gamma_N$,
is isometrically isomorphic to $\msE$,
the latter being
an intrinsic completion of the positive $G$-space
$\Gamma_s(\msD_T)$ in the Krein
$\bigl(\begin{smallmatrix}0 & -\img I^\prime \\
\img I^\prime & 0\end{smallmatrix} \bigr)$-space
$\msH^\prime\op\msH^\prime$, where
$\Gamma_s$ is the operator part of
a boundary relation $\Gamma$ in
a unitary boundary pair $(\msH^\prime,\Gamma)$ for $S^*$,
with the closed domain $\msD_\Gamma=S^*$, and
which in operator case is identified with $\msD_{S^*}$;
such a boundary relation exists by (the proof of)
\cite[Proposition~3.7]{Derkach06}.

Given $\msE$ in \eqref{eq:bvs}
as described above, the operator
$\Gamma_{01}=V\Gamma_s\vrt\msD_T$, where
the operator $V$
is closed continuous and semi-unitary from the
pre-Hilbert space $\Gamma_s(\msD_T)$ to $\msE$,
and which maps $\Gamma_s(\msD_T)$ bijectively onto a dense
lineal in $\msE$;
the existence of $V$ is due to the boundedness of
the form $\gamma_T$ in graph norm.
Letting in particular $\msE=\msD_{\bar{\gamma}_N}$
the operator
\[
\Gamma_{01}=(T+\img I)^{-1}P_\img(T+\img I)
\]
up to unitary equivalence in $\msE$;
$P_\img$ is an orthogonal projection in
$\msH$ onto $\msN_\img$. (In this case the
operator $\sqrt{F_b}=(T+\img I)^{-1}P_\img$
in the proof of \cite[Lemma~3.1]{Brown20}.)

Formalized statements a)--d) are now given in the next
section.
% %%%%%%%%%%%%%%%%%%%%%%%%%%%%%%%%%%%%%%%%%%%%%%%%%%%%%%%%%%%%%%
% %%%%%%%%%%%%%%%%%%%%%%%%%%%%%%%%%%%%%%%%%%%%%%%%%%%%%%%%%%%%%%
% %%%%%%%%%%%%%%%%%%%%%%%%%%%%%%%%%%%%%%%%%%%%%%%%%%%%%%%%%%%%%%
% %%%%%%%%%%%%%%%%%%%%%%%%%%%%%%%%%%%%%%%%%%%%%%%%%%%%%%%%%%%%%%
\section{Main theorem}
% %%%%%%%%%%%%%%%%%%%%%%%%%%%%%%%%%%%%%%%%%%%%%%%%%%%%%%%%%%%%%%
% %%%%%%%%%%%%%%%%%%%%%%%%%%%%%%%%%%%%%%%%%%%%%%%%%%%%%%%%%%%%%%
% %%%%%%%%%%%%%%%%%%%%%%%%%%%%%%%%%%%%%%%%%%%%%%%%%%%%%%%%%%%%%%
% %%%%%%%%%%%%%%%%%%%%%%%%%%%%%%%%%%%%%%%%%%%%%%%%%%%%%%%%%%%%%%
Let $(\msH,[\,\,\cdot\,\,,\,\,\cdot\,\,])$
be a Krein space; an inner
product is conjugate-linear in its first factor.
Let the space of graphs $\msH_\Gamma=\msH\op\msH$
be the Krein space
$(\msH_\Gamma,[\,\,\cdot\,\,,\,\,\cdot\,\,]_\Gamma)$
with an indefinite metric (\cf \cite[Section~2.1]{Azizov89})
\[
[(x_1,y_1)\,,\,(x_2,y_2)]_\Gamma=-\img
([x_1\,,\,y_2]-[y_1\,,\,x_2])\,.
\]

Given a canonical symmetry $J$ in $\msH$,
$(\msH_\Gamma,[\,\,\cdot\,\,,\,\,\cdot\,\,]_\Gamma)$
is a $J_\Gamma$-space with the canonical symmetry
$J_\Gamma(x,y)=(-\img Jy,\img Jx)$.

Given a lineal (\ie a linear subset) $T$ in $\msH_\Gamma$
(\ie a relation in $\msH$), its
$J_\Gamma$-orthogonal complement
(\ie the $J$-adjoint in $\msH$) is denoted by $T^c$;
if $\msH$ is a Hilbert space, $T^c=T^*$.

Although here we mainly work with operators,
it is convenient in many cases to identify
an operator $T$ with its graph, a relation
with trivial indefiniteness (multivalued part)
$\Ind T$.

The domain, kernel, range
of a relation (operator) $T$ are denoted by
$\msD_T$, $\Ker T$ (also $\Ker_\lambda T=\Ker(T-\lambda I)$,
$\lambda\in\bbC$), $\msR_T$
respectively. If necessary,
the notions extend naturally to
relations (operators) acting from one space to another
space.

With an operator $T$ in $\msH$
one associates the symmetric sesquilinear form
\[
\gamma_T[x\,,\,y]=-\img
( [x\,,\,Ty]-[Tx\,,\,y] )\quad\text{on}\quad
\msD_{\gamma_T}=\msD_T
\]
and the corresponding quadratic
form $\gamma_T[x]=\gamma_T[x\,,\,x]$.

An operator $T$ in
$(\msH,[\,\,\cdot\,\,,\,\,\cdot\,\,])$
is dissipative
(symmetric) if $\im[x\,,\,Tx]\geq0$
($\im[x\,,\,Tx]=0$) for all $x\in\msD_T$.
An operator $T$ is dissipative (symmetric) iff
a lineal $T$ in $\msH_\Gamma$,
\ie the graph of $T$, is non-negative
(neutral).

Any dissipative operator can be extended into
a maximal one.
A maximal dissipative operator in $\msH$
(\ie a maximal non-negative subspace in $\msH_\Gamma$)
is necessarily densely defined.
A closed dissipative operator $T$ is maximal
iff such is the operator $(-T^c)$.

Let $(\msH^\prime,
\braket{\,\,\cdot\,\,,\,\,\cdot\,\,}^\prime)$
be a Hilbert space and let the space of graphs
$\msH^\prime_\Gamma=\msH^\prime\op\msH^\prime$
be the Krein space
$(\msH^\prime_\Gamma,
[\,\,\cdot\,\,,\,\,\cdot\,\,]^\prime_\Gamma)$
with an indefinite metric
\[
[(u_1,v_1)\,,\,(u_2,v_2)]^\prime_\Gamma=-\img
(\braket{u_1\,,\,v_2}^\prime-\braket{v_1\,,\,u_2}^\prime)
\]
and canonical symmetry
$J^\prime_\Gamma(u,v)=(-\img v,\img u)$.

The main result is stated as follows.
\begin{thm}\label{thm:1}
A closed dissipative operator in
a Krein space $(\msH,[\,\,\cdot\,\,,\,\,\cdot\,\,])$
with canonical symmetry $J$ has a boundary pair
$(\msE,\Gamma_{01})$; that is,
there is a Hilbert space
$(\msE,\braket{\,\,\cdot\,\,,\,\,\cdot\,\,}_\msE)$
and
an operator $\Gamma_{01}\co\msH\lto \msE$,
$\msD_{\Gamma_{01}}=\msD_T$, which is bounded in graph norm
of $T$, has dense range in $\msE$,
and such that
\begin{equation}
[x\,,\,Ty]-[Tx\,,\,y]=\img
\braket{\Gamma_{01} x\,,\,\Gamma_{01} y}_\msE
\quad(x,y\in\msD_T)\,.
\label{eq:bvs2}
\end{equation}

Specifically:
\begin{SL}
\item[$\mathrm{a)}$]
$\Gamma_{01}$ is of the form
\[
\Gamma_{01}=V\Gamma_sE\,,
\quad
Ex=\begin{pmatrix}x \\ Tx \end{pmatrix}\,,\quad
\msD_E=\msD_T
\]
where $\Gamma_s$ is the operator part
of a boundary relation $\Gamma$
in a ubp $(\msH^\prime,\Gamma)$
for $S^c=\msD_\Gamma$, the
closed symmetric operator $S$ in $\msH$
is defined by
\[
S=T\vrt\Ker\gamma_T
\quad(\Ker\gamma_T=\{x\in\msD_T\vrt
\gamma_T[x]=0\})
\]
and $V$ is a closed
semi-unitary injective operator from the
pre-Hilbert space
$(\Gamma_s(T),[\,\,\cdot\,\,,\,\,\cdot\,\,]^\prime_\Gamma)$
to $\msE$.
\item[$\mathrm{b)}$]
Up to an isometric isomorphism $\msE$
is the completion
$(\msD_{\bar{\gamma}_N},(\,\,\cdot\,\,,\,\,\cdot\,\,)^\sim_N)$
of the pre-Hilbert space
$(\msD_N,(\,\,\cdot\,\,,\,\,\cdot\,\,)_N)$,
where
\[
N=T\vrt(JT+\img I)^{-1}(\msN_\img\cap\msR_{JT+\img I})
\,,\quad
\msN_\img=\Ker_\img JS^c\,,
\]
\[
(x\,,\,y)_N=\gamma_N[x\,,\,y]\quad
(x,y\in\msD_N)\,;\quad
\gamma_N=\gamma_T\quad\text{on}\quad\msD_N
\]
and $\bar{\gamma}_N$ is the closure of
the positive sesquilinear form
$\gamma_N$.
\item[$\mathrm{c)}$]
Particularly therefore one can always choose
the boundary pair with
$\msE=\msD_{\bar{\gamma}_N}$ and
\[
\Gamma_{01}=U(JT+\img I)^{-1}P_\img(JT+\img I)
\]
where $P_\img$ is an orthogonal projection in
$(\msH,[\,\,\cdot\,\,,\,\,J\,\,\cdot\,\,])$
onto $\msN_\img\cap\msR_{JT+\img I}$,
and $U$ is a closed semi-unitary operator
$(\msD_N,(\,\,\cdot\,\,,\,\,\cdot\,\,)_N)\lto
(\msD_{\bar{\gamma}_N},(\,\,\cdot\,\,,\,\,\cdot\,\,)^\sim_N)$,
which maps $\msD_N$ bijectively onto a dense lineal in
$\msD_{\bar{\gamma}_N}$.
\end{SL}
\end{thm}
Any closed symmetric operator $S$ in $\msH$
has a unitary boundary pair (ubp)
$(\msH^\prime,\Gamma)$, where
$\Gamma$ is a unitary relation from
a $J_\Gamma$-space to a $J^\prime_\Gamma$-space,
such that $\bar{\msD}_\Gamma=S^c$;
an overbar refers to the closure.
Without loss of generality one chooses
a ubp with the closed domain, $\msD_\Gamma=S^c$.

If $S$ has equal defect numbers
$\dim\msN_\img=\dim\msN_{-\img}$
($\msN_{-\img}=\Ker_{-\img}JS^c$),
a ubp $(\msH^\prime,\Gamma)$ can be chosen
as an ordinary boundary triple (obt), \ie a ubp
with $\Gamma$ surjective.

The above classic facts on boundary value spaces
of symmetric relations, and much more, can be found
in \cite{Behrndt20,Derkach17,Derkach17b,Derkach12,Behrndt11,Derkach06},
and we skip the details.

For the theory of operators in indefinite inner
product spaces we refer to monographs
\cite{Azizov89,Bognar74}. We only mention
basic notions used in the theorem.
An operator
$U$ from a $W_1$-space
$(\msH_1,[\,\,\cdot\,\,,\,\,\cdot\,\,]_1)$
to a $W_2$-space
$(\msH_2,[\,\,\cdot\,\,,\,\,\cdot\,\,]_2)$
is isometric if
$[Ux_1\,,\,Ux_1]_2=[x_1\,,\,x_1]_1$
($x_1\in\msD_U$);
$U$ is semi-unitary if
it is isometric and $\msD_U=\msH_1$;
$U$ is unitary if it is semi-unitary and
$\msR_U=\msH_2$.
The spaces $\msH_1$ and $\msH_2$ are isometrically
isomorphic if there is a unitary operator
$U$ which is injective
(and hence bijective).

A $W$-space becomes a $G$-space if its Gram
operator has trivial kernel.
The Gram operator of
the $G$-space
$(\msL_T,[\,\,\cdot\,\,,\,\,\cdot\,\,]^\prime_\Gamma)$
in a)
reads
\[
G=P_{\msL_T}J^\prime_\Gamma\vrt\msL_T\,,\quad
\msL_T=\Gamma_s(T)
\]
where $P_{\msL_T}$ is an orthogonal projection
in
$[\,\,\cdot\,\,,\,\,J^\prime_\Gamma\,\,\cdot\,\,]^\prime_\Gamma$
onto $\msL_T$.

The normed spaces $\msH_1$, $\msH_2$
are isomorphic if there is a closed continuous
and continuously invertible bijection,
\ie a homeomorphism, $U\co\msH_1\lto\msH_2$.
Particularly isometrically isomorphic normed
spaces are isomorphic.
Throughout all normed spaces are (pre-)Hilbert spaces
via polarization identity.
% %%%%%%%%%%%%%%%%%%%%%%%%%%%%%%%%%%%%%%%%%%%%%%%%%%%%%%%%%%%%%%
% %%%%%%%%%%%%%%%%%%%%%%%%%%%%%%%%%%%%%%%%%%%%%%%%%%%%%%%%%%%%%%
% %%%%%%%%%%%%%%%%%%%%%%%%%%%%%%%%%%%%%%%%%%%%%%%%%%%%%%%%%%%%%%
% %%%%%%%%%%%%%%%%%%%%%%%%%%%%%%%%%%%%%%%%%%%%%%%%%%%%%%%%%%%%%%
\section{Proofs}
% %%%%%%%%%%%%%%%%%%%%%%%%%%%%%%%%%%%%%%%%%%%%%%%%%%%%%%%%%%%%%%
% %%%%%%%%%%%%%%%%%%%%%%%%%%%%%%%%%%%%%%%%%%%%%%%%%%%%%%%%%%%%%%
% %%%%%%%%%%%%%%%%%%%%%%%%%%%%%%%%%%%%%%%%%%%%%%%%%%%%%%%%%%%%%%
% %%%%%%%%%%%%%%%%%%%%%%%%%%%%%%%%%%%%%%%%%%%%%%%%%%%%%%%%%%%%%%
\begin{lem}\label{lem:1}
Given a subspace $T$ in
$(\msH_\Gamma,[\,\,\cdot\,\,,\,\,\cdot\,\,]_\Gamma)$,
there is a Hilbert space
$(\msH^\prime,
\braket{\,\,\cdot\,\,,\,\,\cdot\,\,}^\prime)$
and a bounded operator
$\bigl(\begin{smallmatrix}
\Gamma_0 \\ \Gamma_1\end{smallmatrix}\bigr)
\in\msB(T,\msH^\prime_\Gamma)$
such that
\begin{equation}
[x\,,\,y^\prime]-[x^\prime\,,\,y]=
\braket{\Gamma_0\hx\,,\,\Gamma_1\hy}^\prime-
\braket{\Gamma_1\hx\,,\,\Gamma_0\hy}^\prime
\label{eq:bvs3}
\end{equation}
for all $\hx=\bigl(\begin{smallmatrix}x \\ x^\prime
\end{smallmatrix}\bigr)\in T$,
$\hy=\bigl(\begin{smallmatrix}y \\ y^\prime
\end{smallmatrix}\bigr)\in T$.

Specifically, if $\Gamma_s$ is the operator
part of a boundary relation $\Gamma$ in a ubp
$(\msH^\prime,\Gamma)$ for $S^c=\msD_\Gamma$,
where $S=T\cap T^c$ is the isotropic part of $T$, then
$\bigl(\begin{smallmatrix}
\Gamma_0 \\ \Gamma_1\end{smallmatrix}\bigr)=\Gamma_s\vrt T$
and is a closed continuous and unitary operator
from the indefinite inner product space
$(T,[\,\,\cdot\,\,,\,\,\cdot\,\,]_\Gamma)$
to the $G$-space
$(\msL_T,
[\,\,\cdot\,\,,\,\,\cdot\,\,]^\prime_\Gamma)$
with
\[
\msL_T=\Gamma_s(T)=\msR_{\bigl(\begin{smallmatrix}
\Gamma_0 \\ \Gamma_1\end{smallmatrix}\bigr)}\,.
\]
\end{lem}
A subspace is a closed lineal.
Observe that in the lemma $T$
is an arbitrary closed relation in $\msH$.
For the column notation of operators we refer to
\cite{Hassi20}.
\begin{proof}
Let $S=T\cap T^c$, the isotropic part of $T$ in $\msH_\Gamma$;
hence there is a ubp
$(\msH^\prime,\Gamma)$ for $S^c=\msD_\Gamma(\supseteq T)$.
That $\Gamma$ is isometric explicitly reads
\[
[\hx\,,\,\hy]_\Gamma=[\hu\,,\,\hv]^\prime_\Gamma
\]
for all $(\hx,\hu)\in\Gamma$, $(\hy,\hv)\in\Gamma$.

That $\Gamma$ is unitary implies particularly it is a closed
relation. Let
$\Gamma_s$ be the operator part of $\Gamma$;
then $\Gamma_s\vrt T$ is the operator part of a closed
relation
$\Gamma\vrt T$, and is therefore a bounded
on $\msD_{\Gamma\vrt T}=T$ isometric operator
$\msH_\Gamma\lto\msH^\prime_\Gamma$.

The range $\Gamma_s(T)$ is a subspace in
$\msH^\prime_\Gamma$, since such is
$\Gamma(T)$.

Since
$\Ind \Gamma$
is an isotropic subspace in $\msR_\Gamma$,
the above Green identity implies
\[
[\hx\,,\,\hy]_\Gamma=[(\Gamma_s\vrt T)\hx\,,
\,(\Gamma_s\vrt T)\hy]^\prime_\Gamma
\]
for all $\hx\in T$, $\hy\in T$.

To accomplish the proof it remains to observe,
first, that the operator $\Gamma_s\vrt T$
can be given the (graph) form
\[
\Gamma_s\vrt T=
\begin{pmatrix}
\Gamma_0 \\ \Gamma_1\end{pmatrix}
\]
where operators
$\Gamma_0=\pi_0\Gamma_s\vrt T$ and
$\Gamma_1=\pi_1\Gamma_s\vrt T$ and
\[
\begin{split}
&\pi_0\co\msH^\prime_\Gamma\lto\msH^\prime\,,\quad
\begin{pmatrix}u \\ u^\prime\end{pmatrix}\mapsto u\,;
\\
&\pi_1\co\msH^\prime_\Gamma\lto\msH^\prime\,,\quad
\begin{pmatrix}u \\ u^\prime\end{pmatrix}\mapsto u^\prime
\end{split}
\]
and, second, that
$(\msL_T,[\,\,\cdot\,\,,\,\,\cdot\,\,]^\prime_\Gamma)$
is a non-degenerate subspace, because
its isotropic part is $\Gamma_s(S)$, while
$S=\Ker\Gamma_s=\Ker\Gamma$.
\end{proof}

Let $T$ be a closed
operator in a $J$-space $\msH$, and consider the Hilbert space
\[
\msH_T=(\msD_T,\braket{\,\,\cdot\,\,,\,\,\cdot\,\,}_T)
\]
with scalar product
\[
\braket{x\,,\,y}_T=\braket{x\,,\,y}+
\braket{Tx\,,\,Ty}\,,\quad
\braket{x\,,\,y}=[x\,,\,Jy]
\]
and graph norm
\[
\norm{x}_T=(\norm{x}^2+\norm{Tx}^2)^{1/2}\,,
\quad\norm{x}=\braket{x\,,\,x}^{1/2}\,.
\]

The form $\gamma_T$ is bounded on $\msH_T$,
$\abs{\gamma_T[x]}\leq2\norm{x}^2_T$,
by Cauchy--Schwarz, so
Reisz's representation theorem tells us that
\[
\gamma_T[x,y]=\braket{x\,,\,Fy}_T
\]
for some self-adjoint operator $F$ in $\msH_T$,
which is bounded, $F\in\msB(\msH_T)$.
We state this as a separate lemma.
\begin{lem}\label{lem:2}
Let $T$ be a closed operator in $\msH$.
Then there is a self-adjoint operator
$F\in\msB(\msH_T)$ such that
\[
[x\,,\,Ty]-[Tx\,,\,y]=
\img\braket{x\,,\,Fy}_T\quad
(x,y\in\msD_T)\,.
\]
\end{lem}

Let $T$ be a closed operator in $\msH$,
$S(=T\vrt\Ker\gamma_T)$ be its symmetric part,
and define the closed operator restriction
\[
N=T\vrt\msH_T\om\msD_S\,.
\]
Note: $T=S\hop N$
(orthogonal componentwise sum);
$\msD_S\cap\msD_N=\{0\}$.

Let $\Gamma_s$ and $\msL_T=\Gamma_s(T)=\Gamma_s(N)$
be as in Lemma~\ref{lem:1}.
Then
\[
\vp=E^{-1}(\Gamma_s\vrt N)^{-1}
\]
is a bijection $\msL_T\lto\msD_N$.

If $T$ in Lemma~\ref{lem:2} is
dissipative,
then $F\geq0$ and
\[
\gamma_T[x]=\norm{\sqrt{F}\,x}^2_T\quad(x\in\msD_T)
\]
and a self-adjoint operator $\sqrt{F}\in\msB(\msH_T)$
is parametrized by
\[
\sqrt{F}=V_F\Gamma_sE\,,\quad
V_F=\sqrt{F}\,\vp\,.
\]
Because
\[
\norm{V_F\hu}^2_T=[\hu\,,\,\hu]^\prime_\Gamma
\quad(\hu\in\msL_T)
\]
the operator $V_F$ with
\[
\msD_{V_F}=\msL_T\,,\quad
\msR_{V_F}=\msR_{\sqrt{F}}
\]
is closed continuous, injective,
and semi-unitary from the positive $G$-space
$(\msL_T,
[\,\,\cdot\,\,,\,\,\cdot\,\,]^\prime_\Gamma)$
to the Hilbert subspace
$(\bar{\msR}_{\sqrt{F}},
\braket{\,\,\cdot\,\,,\,\,\cdot\,\,}_T)$
of $\msH_T$;
$\msL_T$ is non-negative and then positive
in $\msH^\prime_\Gamma$
since $T$ is nonnegative in $\msH_\Gamma$.

With $T$ as previously,
let
$(\msE,\braket{\,\,\cdot\,\,,\,\,\cdot\,\,}_\msE)$
be a Hilbert space isometrically isomorphic to
the Hilbert space
$(\bar{\msR}_{\sqrt{F}},
\braket{\,\,\cdot\,\,,\,\,\cdot\,\,}_T)$,
and let $U_F\co\bar{\msR}_{\sqrt{F}}\lto\msE$
be the (standard) unitary operator.
Put
\[
\Gamma_{01}=U_F\sqrt{F}=
V\Gamma_sE\,,
\quad V=U_FV_F
\]
and note that the lineal $\msR_{\Gamma_{01}}=\msR_V$
is dense in $\msE$.

The latter combined with Lemmas~\ref{lem:1}, \ref{lem:2}
proves, first, the existence of a boundary
pair $(\msE,\Gamma_{01})$ in Theorem~\ref{thm:1}
and, second, the form of $\Gamma_{01}$ in a).
\begin{rems}
\begin{SL}
\item[a)]
As a simple corollary of \eqref{eq:bvs2}:
\\
Real eigenvalues of a closed dissipative operator, $T$,
are precisely those of its symmetric part, $S$.
\\
To see this use
$2(\im\lambda)[x\,,\,x]=\norm{\Gamma_{01}x}^2_\msE$
($x\in\Ker_\lambda T$, $\lambda\in\sigma_p(T)$)
and $\Ker\Gamma_{01}=\msD_S$.
\item[b)]
As it appears from \eqref{eq:bvs2},
in applications it is enough to have established
the domain restriction of $\Gamma_s$ to $T$, which
usually is done via \eqref{eq:bvs3}.
\item[c)]
Consider a pre-Hilbert space
$(\msR_F,(\,\,\cdot\,\,,\,\,\cdot\,\,)_F)$
with scalar product
\[
(u_x\,,\,u_y)_F=\braket{x\,,\,Fy}_T\,;\,
u_x=Fx\,,\,u_y=Fy\,;\,x,y\in\msH_T\,.
\]
Its completion is by \cite[Lemma~2.4]{Curgus03}
a Hilbert space
$(\msR_{\sqrt{F}},(\,\,\cdot\,\,,\,\,\cdot\,\,)^\sim_F)$
with scalar product
\[
(\sqrt{F}x\,,\,\sqrt{F}y)^\sim_F=
\braket{x\,,\,P_Fy}_T\quad
(x,y\in\msH_T)
\]
where $P_F$ is an orthogonal projection in
$\msH_T$ onto $\bar{\msR}_{\sqrt{F}}$,
which is continuously embedded in the Hilbert space
$\msH_T$.
Let $\iota$ be the corresponding embedding
and let $\iota^*\co\msH_T\lto\msR_{\sqrt{F}}$ be its
adjoint, \ie
$(u\,,\,\iota^*x)^\sim_F=\braket{\iota u\,,\,x}_T$
($x\in\msH_T$, $u\in\msR_{\sqrt{F}}$).
Then by \cite[Theorem~2.7]{Curgus03}
$F=\iota\iota^*$.
\end{SL}
\end{rems}

Next we characterize $\msE$ in Theorem~\ref{thm:1} b).

An intrinsic completion
of the positive $G$-space $\msL_T$ in $\msH^\prime_\Gamma$
is the completion
with respect to the intrinsic norm
\[
\abs{\hu}_{\msL_T}=
([\hu\,,\,\hu]^\prime_\Gamma)^{1/2}
\quad(\hu\in\msL_T)\,.
\]
Such a completion of $\msL_T$ is unique in the sense that
any other intrinsic completion
is isometrically isomorphic to the given one,
with an isomorphism which acts as the identity on
$\msL_T$; see \cite{Curgus03} for details.
\begin{lem}\label{lem:3}
$(\msE,\braket{\,\,\cdot\,\,,\,\,\cdot\,\,}_\msE)$
is an intrinsic completion of
$(\msL_T,[\,\,\cdot\,\,,\,\,\cdot\,\,]^\prime_\Gamma)$.
\end{lem}
\begin{proof}
By \cite[Theorem~V.2.1]{Bognar74}
an intrinsic completion of
$(\msL_T,[\,\,\cdot\,\,,\,\,\cdot\,\,]^\prime_\Gamma)$
is a Hilbert space,
$(\msE,\braket{\,\,\cdot\,\,,\,\,\cdot\,\,}_\msE)$ say,
such that there is a
(necessarily closed) semi-unitary
operator $\msL_T\lto\msE$, which maps $\msL_T$
bijectively onto a dense lineal in $\msE$;
$V$ in Theorem~\ref{thm:1} a) has the required
properties.

Given the boundary pair $(\msE,\Gamma_{01})$ of $T$,
any other boundary pair
$(\msE^\prime,\Gamma^\prime_{01})$
is of the form $\msE^\prime=\tV(\msE)$,
$\Gamma^\prime_{01}=\tV\Gamma_{01}$
for a unitary operator
$\tV\co\msE\lto\msE^\prime$;
see \cite[Lemma~3.3]{Brown20} or
Remark~\ref{rem:tV} for an explicit form of $\tV$.
Thus
$\msE^\prime$
is an intrinsic completion
of another positive $G$-space, the latter being
isometrically isomorphic to $\msL_T$.
\end{proof}
An equivalent statement in Lemma~\ref{lem:3}
is that
a Banach space $(\msE,\norm{\,\,\cdot\,\,}_\msE)$
with norm $\norm{\,\,\cdot\,\,}_\msE$
induced by the scalar product
$\braket{\,\,\cdot\,\,,\,\,\cdot\,\,}_\msE$
is a completion of the normed space
$(\msL_T,\abs{\,\,\cdot\,\,}_{\msL_T})$.

The norm
\[
\abs{\,\,\cdot\,\,}_{\msL_T}=\abs{\vp\,\cdot}_N
\quad\text{on}\quad\msL_T
\]
or equivalently
\[
\abs{\,\,\cdot\,\,}_N=\abs{\vp^{-1}\,\cdot}_{\msL_T}
\quad\text{on}\quad\msD_N
\]
where the norm
\[
\abs{\,\,\cdot\,\,}_N=
\sqrt{\gamma_N[\,\,\cdot\,\,]}\quad
(\text{with}\quad
\gamma_N=\gamma_T\quad\text{on}\quad
\msD_{\gamma_N}=\msD_N)\,.
\]
That is,
$\vp$ is an isometric isomorphism
from the normed space
$(\msL_T,\abs{\,\,\cdot\,\,}_{\msL_T})$ to the normed
space $(\msD_N,\abs{\,\,\cdot\,\,}_N)$,
which therefore
extends to an isometric isomorphism
from a completion $(\msE,\norm{\,\,\cdot\,\,}_\msE)$
of $(\msL_T,\abs{\,\,\cdot\,\,}_{\msL_T})$ to a completion
of $(\msD_N,\abs{\,\,\cdot\,\,}_N)$.

The next lemma accomplishes the proof of
Theorem~\ref{thm:1} b).
\begin{lem}
$(\msE,\norm{\,\,\cdot\,\,}_\msE)$
is isometrically isomorphic to a completion
of the normed space $(\msD_N,\abs{\,\,\cdot\,\,}_N)$.

Moreover, the space
$(\msD_{\bar{\gamma}_N},\abs{\,\,\cdot\,\,}^\sim_N)$
with norm
\[
\abs{x}^\sim_N=\sqrt{\bar{\gamma}_N[x]}\quad
(x\in\msD_{\bar{\gamma}_N})
\]
is the completion of
$(\msD_N,\abs{\,\,\cdot\,\,}_N)$.
\end{lem}
\begin{proof}
The form $\gamma_N$ is closable, since
$\gamma_N[x]\leq2\norm{x}\,\norm{Nx}$
($x\in\msD_{\gamma_N}$).

Since $(\msD_{\bar{\gamma}_N},
\norm{\,\,\cdot\,\,}_{\bar{\gamma}_N})$
is the Banach space with norm
$\norm{x}_{\bar{\gamma}_N}=
(\abs{x}^{\sim\,2}_N+\norm{x}^2)^{1/2}$
($x\in\msD_{\bar{\gamma}_N}$), each sequence
$(x_n)$ in $\msD_{\bar{\gamma}_N}$ converges to
$x\in\msD_{\bar{\gamma}_N}$ in norm
$\norm{\,\,\cdot\,\,}_{\bar{\gamma}_N}$, and then also in
norm $\abs{\,\,\cdot\,\,}^\sim_N$. In particular therefore
a Cauchy sequence $(x_n)$
in $(\msD_{\bar{\gamma}_N},\abs{\,\,\cdot\,\,}^\sim_N)$
is convergent.

That
$(\msD_{\bar{\gamma}_N},\abs{\,\,\cdot\,\,}^\sim_N)$
is the completion of
$(\msD_{\gamma_N},\abs{\,\,\cdot\,\,}_N)$
is seen from the definition:
$\msD_{\bar{\gamma}_N}$ is the set of limits in
$(\msH,\norm{\,\,\cdot\,\,})$ of
Cauchy sequences in
$(\msD_{\gamma_N},\abs{\,\,\cdot\,\,}_N)$; \ie
$\msD_{\gamma_N}$ is dense in
$(\msD_{\bar{\gamma}_N},\abs{\,\,\cdot\,\,}^\sim_N)$
and moreover
$\abs{\,\,\cdot\,\,}^\sim_N=\abs{\,\,\cdot\,\,}_N$
on $\msD_{\gamma_N}$.
\end{proof}
\begin{rems}
As a simple corollary:
\begin{SL}
\item[a)]
$\dim\msE=\dim\msD_{\bar{\gamma}_N}$; hence
$(\msL_T,\abs{\,\,\cdot\,\,}_{\msL_T})$ is
a Banach space iff $\gamma_N=\bar{\gamma}_N$.
\item[b)]
The normed spaces
$(\msR_{\sqrt{F}},\norm{\,\,\cdot\,\,}_T)$
and
$(\msD_N,\abs{\,\,\cdot\,\,}_N)$ and
$(\msL_T,\abs{\,\,\cdot\,\,}_{\msL_T})$
are pairwise isometrically isomorphic.
\\
Since $\Ker\sqrt{F}=\msD_S$ and since
$\msD_S\cap\msD_N$ is trivial,
$(\sqrt{F}\vrt\msD_N)^{-1}$ is such an isometric
isomorphism $\msR_{\sqrt{F}}\lto\msD_N$
by $\gamma_N=\gamma_T$ on
$\msD_N$.
\end{SL}
\end{rems}
The above Banach spaces are Hilbert spaces
via $\gamma_N[\,\,\cdot\,\,,\,\,\cdot\,\,]$.
Particularly
$(\msD_{\bar{\gamma}_N},(\,\,\cdot\,\,,\,\,\cdot\,\,)^\sim_N)$
with scalar product
$(\,\,\cdot\,\,,\,\,\cdot\,\,)^\sim_N=
\bar{\gamma}_N[\,\,\cdot\,\,,\,\,\cdot\,\,]$
is the Hilbert space.

Lastly, in order to see Theorem~\ref{thm:1} c)
use that $\vp\Gamma_sE$ is $0$ on $\msD_S$
and $I$ on $\msD_N$, and then that
$\msD_N=(JT+\img I)^{-1}(\msN_\img\cap\msR_{JT+\img I})$.
\begin{rem}\label{rem:tV}
With the boundary pair in Theorem~\ref{thm:1} c)
$U$ extends by continuity to a unitary operator in
$(\msD_{\bar{\gamma}_N},(\,\,\cdot\,\,,\,\,\cdot\,\,)^\sim_N)$,
which can be set to $I$ without loss of generality;
and similarly for $V$ in a).
Then the operator $\tV=\vp^{\prime\,-1}\vp$
in the proof of Lemma~\ref{lem:3}
is uniquely determined by an extension
$\vp\co\msE\lto\msD_{\bar{\gamma}_N}$
and a similarly defined extension
$\vp^\prime\co\msE^\prime\lto\msD_{\bar{\gamma}_N}$ for
the boundary pair $(\msE^\prime,\Gamma^\prime_{01})$.
\end{rem}
% %%%%%%%%%%%%%%%%%%%%%%%%%%%%%%%%%%%%%%%%%%%%%%%%%%%%%%%%%%%%%%
% %%%%%%%%%%%%%%%%%%%%%%%%%%%%%%%%%%%%%%%%%%%%%%%%%%%%%%%%%%%%%%
% %%%%%%%%%%%%%%%%%%%%%%%%%%%%%%%%%%%%%%%%%%%%%%%%%%%%%%%%%%%%%%
% %%%%%%%%%%%%%%%%%%%%%%%%%%%%%%%%%%%%%%%%%%%%%%%%%%%%%%%%%%%%%%
\section{Criterion for completeness}
% %%%%%%%%%%%%%%%%%%%%%%%%%%%%%%%%%%%%%%%%%%%%%%%%%%%%%%%%%%%%%%
% %%%%%%%%%%%%%%%%%%%%%%%%%%%%%%%%%%%%%%%%%%%%%%%%%%%%%%%%%%%%%%
% %%%%%%%%%%%%%%%%%%%%%%%%%%%%%%%%%%%%%%%%%%%%%%%%%%%%%%%%%%%%%%
% %%%%%%%%%%%%%%%%%%%%%%%%%%%%%%%%%%%%%%%%%%%%%%%%%%%%%%%%%%%%%%
Let $T$ be a closed dissipative operator in a Krein space $\msH$.
When is the form $\gamma_N$ closed?
Some of the equivalent conditions are:
$\msR_G=\msL_T$;
$\msL_T$ is uniformly positive;
$\msL_T[\hsum]\msL^{[\bot]}_T=\msH^\prime_\Gamma$.
If at least one holds,
the Hilbert space
$(\msL_T,[\,\,\cdot\,\,,\,\,\cdot\,\,]^\prime_\Gamma)
=(\msE,\braket{\,\,\cdot\,\,,\,\,\cdot\,\,}_\msE)$.

For example, in Theorem~\ref{thm:1} a)
\[
\msR_{\Gamma_{01}}=
\msR_V=\msE\quad\text{if}\quad
\dim\msH^\prime<\infty
\]
since a finite-dimensional
non-degenerate subspace is a Krein space
(\cf \cite[Lemma~2.4]{Baidiuk21});
in our case $\msL_T$ is positive, hence a Hilbert space.

The Hilbert space $\msH^\prime$ is finite-dimensional
if \eg $S$ has equal finite defect numbers.

In practice one vaguely determines
$\Gamma_0$ and $\Gamma_1$, and then rigorously verifies
\eqref{eq:bvs3}. Therefore
suppose $\Gamma_0$, $\Gamma_1$ are already
known: We give a necessary and sufficient
condition for $\gamma_N$ to be closed in terms of
these operators.

Define operator restrictions $S_1$, $N_1$ of $T$,
which satisfy $T=S_1\hop N_1$,
by
\[
S_1=\Ker\Gamma_1\,,\quad
N_1=T\vrt\msH_T\om\msD_{S_1}\,.
\]
Define also the operator $\Theta_{01}$
in $\msH^\prime$ by
\[
\Theta_{01}=(\Gamma_0\vrt N_1)
(\Gamma_1\vrt N_1)^{-1}\quad
(\msD_{\Theta_{01}}=\msR_{\Gamma_1})\,.
\]
It is a closed operator, because
the inverse relation $\Theta^{-1}_{01}=\Gamma_s(N_1)$
in $\msH^\prime$ is a subspace in $\msH^\prime_\Gamma$;
$\Theta^{-1}_{01}$ is an operator iff
$S_0\cap N_1=\{0\}$, where
$S_0=\Ker\Gamma_0$.

The adjoint $\Theta^*_{01}$ in $\msH^\prime$
of $\Theta_{01}$ is
a closed relation in $\msH^\prime$
($\Ind\Theta^*_{01}=\msR^\bot_{\Gamma_1}$) with
dense domain.
\begin{prop}\label{prop:cmpl}
The following are equivalent:
\begin{SL}
\item[$\mathrm{a)}$]
$(\msR_{\bigl(\begin{smallmatrix}
\Gamma_0 \\ \Gamma_1\end{smallmatrix}\bigr)},
[\,\,\cdot\,\,,\,\,\cdot\,\,]^\prime_\Gamma)$
is complete, \ie the form $\gamma_N$ is closed;
\item[$\mathrm{b)}$]
$\norm{[ (\Gamma_1-\img\Gamma_0)\vrt N ]
[ (\Gamma_1+\img\Gamma_0)\vrt N ]^{-1}}<1$;
\item[$\mathrm{c)}$]
$\msR_{\Gamma_1}+(\Gamma_0(S_1))^\bot=\msH^\prime=
(\Theta^*_{01}-\Theta_{01})
( (\Gamma_0(S_1))^\bot )+\Gamma_0(S_1)$.
\end{SL}
\end{prop}
In b) $\norm{\,\,\cdot\,\,}$ is the operator
sup-norm, and in c)
the action of the relation
$\Theta^*_{01}-\Theta_{01}$ on the orthogonal
complement $(\Gamma_0(S_1))^\bot$ means
the action on the lineal
$\msD_{\Theta_{01}}\cap\msD_{\Theta^*_{01}}\cap
(\Gamma_0(S_1))^\bot$.
\begin{proof}
a) $\Leftrightarrow$ b) This is the statement that
a positive subspace $\msL_T$ in $\msH^\prime_\Gamma$
with the angular operator $K$ is uniformly positive iff
$\norm{K}<1$. It remains to remark, first, that
$K$ in $\msH^\prime_\Gamma$ with domain
$\img I\vrt\msR_{\Gamma_1+\img\Gamma_0}$ maps
$\bigl(\begin{smallmatrix}u \\ \img u\end{smallmatrix} \bigr)$
to
$\bigl(\begin{smallmatrix}2\img v-u \\
\img u+2v\end{smallmatrix} \bigr)$, where
$u=(\Gamma_1+\img\Gamma_0)\hx$,
$v=\Gamma_0\hx$, $\hx\in T$,
and, second, that $\Ker(\Gamma_1+\img\Gamma_0)=S$.

a) $\Leftrightarrow$ c)
Viewing $\msL_T$ as a relation in $\msH^\prime$,
the equality in
$\msL_T[\hsum]\msL^{[\bot]}_T\subseteq\msH^\prime_\Gamma$
(where $\msL^{[\bot]}_T$ is the $J^\prime_\Gamma$-orthogonal
complement of $\msL_T$)
holds iff
\\
$\msR_{\msL_T}+\msR_{\msL^{[\bot]}_T}=\msH^\prime=
\Ker(\msL_T[\hsum]\msL^{[\bot]}_T)$.

Because
$\Ker(\msL_T[\hsum]\msL^{[\bot]}_T)=
\msR_{ (\msL^{[\bot]}_T)^{-1}\msL_T-I }$, so
it suffices to show that
\[
\msL^{[\bot]}_T=[\Theta^*_{01}\vrt\msD_{\Theta^*_{01}}\cap
(\Gamma_0(S_1))^\bot]^{-1}
\]
for the rest is a routine computation.

But the above $\msL^{[\bot]}_T$ is easily seen from
$(u,v)\in\msL_T$ $\Leftrightarrow$
$u=\Theta_{01}v+\xi$, some $v\in\msR_{\Gamma_1}$
and $\xi\in\Gamma_0(S_1)$.
\end{proof}
\begin{rem}
If \eg
$(\msH^\prime,\Gamma)$ is an obt
for $S^c$, this does not necessarily imply
$\msR_{\Gamma_1}=\msH^\prime$, since by definition
$\Gamma_1$ is $\pi_1\Gamma$ restricted to $T\subseteq S^c$.
\end{rem}
% %%%%%%%%%%%%%%%%%%%%%%%%%%%%%%%%%%%%%%%%%%%%%%%%%%%%%%%%%%%%%%
% %%%%%%%%%%%%%%%%%%%%%%%%%%%%%%%%%%%%%%%%%%%%%%%%%%%%%%%%%%%%%%
% %%%%%%%%%%%%%%%%%%%%%%%%%%%%%%%%%%%%%%%%%%%%%%%%%%%%%%%%%%%%%%
% %%%%%%%%%%%%%%%%%%%%%%%%%%%%%%%%%%%%%%%%%%%%%%%%%%%%%%%%%%%%%%
\section{Example of application}
% %%%%%%%%%%%%%%%%%%%%%%%%%%%%%%%%%%%%%%%%%%%%%%%%%%%%%%%%%%%%%%
% %%%%%%%%%%%%%%%%%%%%%%%%%%%%%%%%%%%%%%%%%%%%%%%%%%%%%%%%%%%%%%
% %%%%%%%%%%%%%%%%%%%%%%%%%%%%%%%%%%%%%%%%%%%%%%%%%%%%%%%%%%%%%%
% %%%%%%%%%%%%%%%%%%%%%%%%%%%%%%%%%%%%%%%%%%%%%%%%%%%%%%%%%%%%%%
Let $\Omega$ be a measurable subset of $[0,\infty)$,
which is not a null set, and
let $\iota$ be the continuous inclusion of
$L^2(\Omega)$ in $L^2(0,\infty)$:
$\ty=\iota y=y$ a.e.\,on $\Omega$ and
$\ty=0$ a.e.\,on $\Omega^c=[0,\infty)\setm\Omega$.
Define the lineal
$\msL_\Omega=\{y\in L^2(\Omega)\vrt\ty\in H^2(0,\infty)\,;\,
\ty^\prime(0)=h\ty(0) \}$, where $h$ is a real number.
Let $q$ be a measurable bounded
complex-valued function on $\Omega$, with
$\im q(t)>0$ for a.e.\,$t\in\Omega$.
\begin{prop}
Let $L=(-\frac{d^2}{dt^2}+q(t))\vrt\msL_\Omega$,
a closed densely defined dissipative operator in
$L^2(\Omega)$. Then the operator norm of its Cayley transform
$\norm{(L-\img I)(L+\img I)^{-1}}=1$.
\end{prop}
For a dissipative $L$, clearly
$(L-\img I)(L+\img I)^{-1}$ is contractive,
so the point here is that the sup-norm is precisely $1$.
\begin{proof}
Denote by the same $q$ a trivial extension of $q$ to
$[0,\infty)$; one can equally assume that
$q$ on $[0,\infty)$ is as in
\cite[Examples~3.4, 4.14]{Brown20}.
In the Hilbert space $\msH=L^2(0,\infty)$
the maximal dissipative operator
$T=\tau=-\frac{d^2}{dt^2}+q(t)$ on
$x\in H^2(0,\infty)$, $x^\prime(0)=hx(0)$.
The quadratic form
$\gamma_T[x]=
2\norm{\sqrt{\im q}\,x}^2_{L^2(\Omega)}$
$(x\in\msD_T)$,
so
$\msD_S=\Ker\gamma_T$ reads
$\{x\in\msD_T\vrt
x=0\,\text{a.e.\,on}\,\Omega \}$, \ie
$\msD_S=\iota_c(\msL^\prime_{\Omega^c})$,
$\msL^\prime_{\Omega^c}=\{y\in L^2(\Omega^c)\vrt
\iota_c y\in\msD_T \}$, where $\iota_c$ is the continuous
inclusion of
$L^2(\Omega^c)$ in $L^2(0,\infty)$:
$\iota_c y=0$ a.e.\,on $\Omega$ and
$\iota_c y=y$ a.e.\,on $\Omega^c$. Then
$\msD_N=\iota(\msL_\Omega)$. In order to see the latter,
let $\Delta_N=\{x\in\msD_T\vrt
x=0\,\text{a.e.\,on}\,\Omega^c \}$. On the one
hand
$\msD_S\cap\Delta_N=\{0\}$, $\msD_S\dsum\Delta_N=\msD_T$;
on the other hand
$\msD_S\dsum\msD_N=\msD_T$,
$\Delta_N\subseteq\msD_N=
\msH_T\om\msD_S$, so $\Delta_N=\msD_N$
with
\[
(\iota y\,,\,\iota y_1)_N=
\gamma_L[y,y_1]\quad
(y,y_1\in\msL_\Omega)\,.
\]

Since $\msD_T\subseteq\msH$ densely,
$\msL_\Omega\subseteq L^2(\Omega)$ densely;
hence
$\msD_{\bar{\gamma}_N}=\iota(L^2(\Omega))$
with
\[
(\iota y\,,\,\iota y_1)^\sim_N=
\braket{y\,,\,y_1}_{L^2(\Omega,2\im q(t)dt)}\quad
(y,y_1\in L^2(\Omega))\,.
\]
(Because the norm
$\norm{\,\,\cdot\,\,}_{L^2(\Omega,2\im q(t)dt)}=
\norm{\sqrt{2\im q}\,\,\cdot\,}_{L^2(\Omega)}$
is equivalent to
$\norm{\,\,\cdot\,\,}_{L^2(\Omega)}$, the Hilbert
space
$(\msD_{\bar{\gamma}_N},
(\,\,\cdot\,\,,\,\,\cdot\,\,)^\sim_N)$
is isomorphic to $L^2(\Omega)=L^2(\Omega,dt)$.)

From the above,
since
$(\iota y\,,\,\iota y_1)_N=(\iota y\,,\,\iota y_1)^\sim_N$
for $y$, $y_1\in\msL_\Omega$, one puts in \eqref{eq:bvs3}
$\Gamma_0\hx=y$, $\Gamma_1\hx=Ly$, where
$\hx=Ex$, $x=x_S+\iota y$, $x_S\in\msD_S$, $y\in\msL_\Omega$,
and then applies Proposition~\ref{prop:cmpl} a), b),
which accomplishes the proof.
\end{proof}
\begin{rem}
In the above, the boundary pair
$(\msE,\Gamma_{01})$ of $T$ reads
$\msE=\msD_{\bar{\gamma}_N}$,
$\Gamma_{01}x=x\vrt\Omega$ ($x\in\msD_T$);
$\Gamma_{01}=V
\bigl(\begin{smallmatrix}\Gamma_0 \\
\Gamma_1 \end{smallmatrix}\bigr)E$ with
$V=\vp\co \msL_T=L\lto\msD_N=\iota(\msL_\Omega)$,
$\bigl(\begin{smallmatrix}y \\
Ly \end{smallmatrix}\bigr)\mapsto\iota y=x\vrt\Omega$.
For illustrative purposes we sketch
the construction of
a ubp $(\msH^\prime,\Gamma)$ for $S^*$.

The operator
$H\co L^2(\Omega)\op L^2(\Omega^c)\lto\msH$,
$\bigl(\begin{smallmatrix}y \\ y_c
\end{smallmatrix}\bigr)\mapsto
\iota y+\iota_cy_c$
defines an isometric isomorphism, and
$H^{-1}TH=\bigl(\begin{smallmatrix}L & 0 \\ 0 & L_c
\end{smallmatrix}\bigr)$ in $L^2(\Omega)\op L^2(\Omega^c)$,
where $L_c=\tau\vrt\msL^\prime_{\Omega^c}$ in
$L^2(\Omega^c)$ is closed densely defined symmetric.
Let $L^*_c$ denote the adjoint operator
and let $(\msH^{\prime\prime},\Gamma^\prime)$
be a ubp for $L^*_c=\msD_{\Gamma^\prime}$.
Then
the adjoint relation $S^*$ in $\msH$ is given by
\[
H^{-1}S^*H=\begin{pmatrix}
L^2(\Omega)\op L^2(\Omega) & 0 \\ 0 & L^*_c\end{pmatrix}
\]
and a ubp $(\msH^\prime,\Gamma)$
for $S^*=\msD_\Gamma$ can be given the form:
$\msH^\prime=\msH^{\prime\prime}\op L^2(\Omega)$ with
\[
\Gamma=\left\{
\left( 
\left( H\begin{pmatrix}y \\ y_c \end{pmatrix},
H\begin{pmatrix}y_1 \\ L^*_cy_c \end{pmatrix}\right),
\left( 
\begin{pmatrix}u \\ y \end{pmatrix},
\begin{pmatrix}v \\ y_1 \end{pmatrix}
\right) \right)\Bigl\vert
y,y_1\in L^2(\Omega)\,; \right.
\]
\[
\left.
\left(\begin{pmatrix}y_c \\ L^*_cy_c \end{pmatrix},
\begin{pmatrix}u \\ v \end{pmatrix} \right)\in
\Gamma^{\prime}
\right\}\,.
\]
By replacing $\Gamma^\prime$ by its operator part
one constructs the operator part $\Gamma_s$ of $\Gamma$;
hence the operator
$\Gamma_s\vrt T=
\bigl(\begin{smallmatrix}
\Gamma_0 \\ \Gamma_1\end{smallmatrix}\bigr)$
maps $\bigl( H\bigl(\begin{smallmatrix}y \\ y_c
\end{smallmatrix}\bigr),
H\bigl(\begin{smallmatrix}Ly \\ L_cy_c
\end{smallmatrix}\bigr)\bigr)\in T$
to $\bigl(\bigl(\begin{smallmatrix}0 \\ y
\end{smallmatrix} \bigr),
\bigl(\begin{smallmatrix}0 \\ Ly
\end{smallmatrix} \bigr) \bigr)\in\msL_T$.

Similar analysis can be done for complex $h$, $\im h>0$,
by replacing $L^2(\Omega,2\im q(t)dt)$ by
$L^2(\Omega\cup\{0\},2\im h\,\delta+2\im q(t)dt)$,
where $\delta$ is the Dirac measure supported at $0$.
\end{rem}


\end{document}